\documentclass[reqno,12pt]{amsart}

\usepackage{biblatex}
\usepackage[pdftex]{hyperref}

\usepackage{amsmath,amsfonts,amssymb}
\usepackage{orcidlink}
\usepackage{fullpage}
\usepackage{enumitem}
\usepackage{xcolor}
\usepackage{algorithm}
\usepackage[noend]{algpseudocode}

\newtheorem{lemma}{Lemma}
\newtheorem{theorem}{Theorem}
\newtheorem{problem}{Problem}
\newtheorem{question}{Question}

\theoremstyle{definition}
\newtheorem{remark}{Remark}

\newcommand{\tKP}{\mathbf{KP}}
\newcommand{\tSSP}{\mathbf{SSP}}

\newcommand{\Mat}{\mathrm{Mat}}
\newcommand{\NP}{\mathsf{NP}}
\newcommand{\MP}{\mathcal{M}_{\max, {+}}^k}
\newcommand{\MT}{\mathcal{M}_{\max, {\times}}^k}
\newcommand{\RP}{\mathcal{R}_{\max, {+}}^k}
\newcommand{\QP}{\mathcal{Q}_{\max, {+}}^k}
\newcommand{\ZMP}{\mathcal{N}_{\max, {+}}}
\newcommand{\ZMT}{\mathcal{N}_{\max, {\times}}}
\newcommand{\ZP}{\mathcal{N}_{+}}
\newcommand{\ZT}{\mathcal{N}_{\times}}
\newcommand{\SSP}{\mathbf{SSP}(\ZP)}
\newcommand{\XThreeC}{\mathbf{X3C}}
\newcommand{\KP}{\mathbf{KP}(\ZP)}
\newcommand{\MKP}{\mathbf{KP}(\ZT)}
\newcommand{\SPP}{\mathbf{SSP}(\ZT)}
\newcommand{\SSMPP}{\mathbf{SSP}(\MP)}
\newcommand{\SSMTP}{\mathbf{SSP}(\MT)}
\newcommand{\KMPP}{\mathbf{KP}(\MP)}
\newcommand{\KMTP}{\mathbf{KP}(\MT)}
\DeclareMathOperator{\size}{size}
\DeclareMathOperator{\V}{\mathrm{V}}

\title{On tropical knapsack-type problems}

\author{I. M. Buchinskiy \and M. V. Kotov \and A. V. Treier}

\address{I. M. Buchinskiy\,\orcidlink{0000-0001-8637-9127}, Sobolev Institute of Mathematics of SB RAS, Omsk, Russia}
\email{buchvan@mail.ru}

\address{M. V. Kotov\,\orcidlink{0000-0002-2820-1053}, Sobolev Institute of Mathematics of SB RAS,  Omsk, Russia}
\email{matvej.kotov@gmail.com}

\address{A. V. Treier\,\orcidlink{0000-0003-4896-1620}, Sobolev Institute of Mathematics of SB RAS, Omsk, Russia; Caucasus Mathematical Center of Adyghe State University, Maykop, Russia}
\email{alexander.treyer@gmail.com}

\date{\today}

\begin{document}

\begin{abstract}
    In this paper, we investigate the computational complexity of the knapsack problem and subset sum problem for the following tropical algebraic structures.
    We consider the semigroup of square matrices of size $k \times k$ with non-negative entries over the max-plus algebra and the semigroup square matrices of size $k \times k$ with positive entries over the max-times algebra.
    We prove that the knapsack problem and subset sum problem for these structures are $\textsf{NP}$-complete.
    We demonstrate that there are pseudo-polynomial algorithms to solve these problems.
    Also, we show that for the latter semigroup, there are polynomial generic algorithms to solve the knapsack problem and the subset sum problem.
\end{abstract}

\maketitle

\section{Introduction}
The 0--1 knapsack problem, the unbounded knapsack problem, and the subset sum problem are well-known problems in combinatorial optimization.
The \textit{0--1 knapsack problem} can be formulated as follows.
\begin{problem}\label{P1}
    Given a number $c$ and $n$ items, each item $i$ has a profit $p_i$ and weight $w_i$, maximize $\sum_{i = 1}^n p_i x_i$ subject to $\sum_{i = 1}^n w_i x_i \leq c$, $x_i \in \{0, 1\}$ for $i \in \{1, \ldots, n\}$.
\end{problem}

When the profit and the weight of each item are identical, the corresponding problem is called the \textit{subset sum problem}:
\begin{problem}\label{P2}
    Given a number $c$ and $n$ items, each item $i$ has a weight $w_i$, maximize $\sum_{i = 1}^n w_i x_i$ subject to $\sum_{i = 1}^n w_i x_i \leq c$, $x_i \in \{0, 1\}$ for $i \in \{1, \ldots, n\}$. 
\end{problem}

Another variant is that each item can be chosen multiple times. The \textit{unbounded knapsack problem} can be formulated as follows:
\begin{problem}\label{P3}
    Given a number $c$ and $n$ items, each item $i$ has a profit $p_i$ and weight $w_i$, maximize $\sum_{i = 1}^n p_i x_i$ subject to $\sum_{i = 1}^n w_i x_i \leq c$, $x_i \in \mathbb{Z}_{\geq 0}$ for $i \in \{1, \ldots, n\}$.
\end{problem}

Also, one can consider decision versions of such problems. 

\begin{problem}\label{ClassicalSSP}
    Given a number $c$ and $n$ items, each item $i$ has weight $w_i$, decide if $x_1 w_1 + \ldots + x_n w_n = c$ for some $x_1, \ldots, x_n \in \{0, 1\}$.
\end{problem}

\begin{problem}\label{ClassicalKnapsack}
    Given a number $c$ and $n$ items, each item $i$ has weight $w_i$, decide if $x_1 w_1 + \ldots + x_n w_n = c$ for some $x_1, \ldots, x_n \in \mathbb{Z}_{\geq 0}$.
\end{problem}

There are other variations of Problems~\ref{P1}---\ref{ClassicalKnapsack}. Knapsack-type problems have been studied for many years. 
Many of such problems are $\NP$-complete, but often there are pseudo-polynomial algorithms based on dynamical programming. 
A good overview of the results in this area can be found in two papers by Cacchiani, Iori, Locatelli, Martello~\cite{CacchianiIoriLocatelliMartello2022_1, CacchianiIoriLocatelliMartello2022_2}.

Miasnikov, Nikolaev, and Ushakov~\cite{MiasnikovNikolaevUshakov2015} generalized Problem~\ref{ClassicalSSP} and Problem~\ref{ClassicalKnapsack} to arbitrary groups.

\begin{problem}\label{Problem6}
    Let $\mathcal{G}$ be a group. Given $w_1, \ldots, w_n, c\in \mathcal{G}$, decide if $w_1^{x_1} \ldots w_k^{x_n} = c$ for some non-negative integers $x_1, \ldots, x_n$.
\end{problem}

\begin{problem} \label{Problem7}
    Let $\mathcal{G}$ be a group. Given $w_1, \ldots, w_n, c\in \mathcal{G}$, decide if $w_1^{x_1} \ldots w_k^{x_n} = c$ for some $x_1, \ldots, x_n \in \{0, 1\}$.
\end{problem}

They used the following terminology.
Problem~\ref{Problem6} is called the \textit{subset sum problem for $\mathcal{G}$} and denoted by $\mathbf{SSP}(\mathcal{G})$.
Problem~\ref{Problem7} is called the \textit{knapsack problem for $\mathcal{G}$} and denoted by $\mathbf{KP}(\mathcal{G})$.
These problems can be generalized to any semigroup.

Studying the knapsack problem and the subset sum problem for groups has become a very active area of research.
This area is sometimes called non-commutative discrete optimization. Let us give a very brief overview of the scope of tris research.
In the pioneering paper in this field ~\cite{MiasnikovNikolaevUshakov2015}, Miasnikov, Nikolaev, and Ushakov stated some knapsack-type problems in groups and proved many initial results: they proved that $\tSSP(\mathcal{G})$ is $\NP$-complete in free metabelian groups, metabelian Baumslag--Solitar groups, $\mathbb{Z} \wr \mathbb{Z}$, and the Thompson group.
Also, they proved that $\tSSP(\mathcal{G})$ is polynomial in nilpotent groups and hyperbolic groups (see also \cite{Lohrey2020} for results on $\tKP(\mathcal{G})$ in hyperbolic groups).
In contrast with the polynomial complexity- of the subset sum problem for nilpotent groups, the knapsack problem is undecidable for them (see \cite{KonigLohreyZetzsche2016, MishchenkoTreier2017} and also \cite{MishchenkoTreier2018KP}, where the authors constructed a so-called universal input for the knapsack problem in nilpotent groups).
In \cite{FrenkelNikolaevUshakov2016}, the knapsack problem and the subset sum problem in direct and free products of groups were studied.
The knapsack problem was also intensively studied for partially commutative groups \cite{LohreyZetzsche2016, LohreyZetzsche2017}.
There are quite a lot of interesting results on $\tKP(\mathcal{G})$ and $\tSSP(\mathcal{G})$ in solvable groups.
It was proved that the subset sum problem is strongly $\NP$-complete for the lamplighter group $L = \mathbb{Z}_2 \wr \mathbb{Z}$, which implies the same result for a wide class of groups that contain $L$ as a subgroup (see \cite{MishchenkoTreier2018SSP} and \cite{GanardiKonigLohreyZetzsche2018}, where authors proved additionally the decidability of the knapsack problem for free solvable groups and stated results on the complexity of the knapsack problem in wreath products).
In \cite{BergstrasserGanardiZetzsche2021}, the conditions when the knapsack problem is decidable in wreath products were presented.
The knapsack problem in solvable non-metabelian groups also attracted the attention of researchers. In the papers \cite{DudkinTreier2018, DudkinTreier2018RUS}, it was proved that the knapsack problem is solvable for the Baumslag--Solitar groups $BS(m, n)$ where the $m \neq 1$ and $n \neq 1$ are relatively prime integers.
The decidability of the knapsack problem for $\mathrm{gcd}(m,n) > 1$ remains open.
As was mentioned before, it was shown in \cite{GanardiKonigLohreyZetzsche2018} that the knapsack problem is solvable in free solvable groups.
Several interesting results were obtained on the decidability and the algorithmic complexity of the knapsack problem in solvable Baumslag-Solitar groups in \cite{GanardiLohreyZetzsche2023}.
Ushakov~\cite{Ushakov2024} recently considered a connection between the problem of finding solutions for a certain constrained spherical equation and the subset sum problem for some class of groups.

Rybalov studied the knapsack problem and the subset sum problem for some matrix semigroups~\cite{Rybalov2023,Rybalov2020, Rybalov2020_2}. He proved that the knapsack problem and the subset sum problem for $\Mat_k(\mathbb{Z}_{\geq0})$ are $\NP$-complete and decidable generically in polynomial time. We will use some of Rybalov's methods in our paper.

Grigoriev and Shpilrain~\cite{GrigorievShpilrain2014,GrigorievShpilrain2019} and Durcheva~\cite{Durcheva2014} suggested using matrix semirings over tropical algebraic structures in cryptography.
For the tropical max-plus algebra, we can consider the set of matrices of size $k\times k$ equipped with matrix addition and matrix multiplication defined in the obvious way. 
It turned out that the behavior of powers of tropical matrices should be studied to analyze these protocols~\cite{KotovUshakov2018,IsaacKahrobaei2021, BuchinskiyKotovTreier2023_2, BuchinskiyKotovTreier2023_1, BuchinskiyKotovTreier2024}.

Muanalifah and Sergeev~\cite{MuanalifahSergeev2022} suggested considering the following tropical discrete logarithm problem.
\begin{problem}
    Given matrices $V, W$, and $C$ over the max-plus algebra of appropriate dimensions, find $x \in \mathbb{Z}_{\geq 0}$ such that $V \otimes W^{\otimes x} = C$.
\end{problem}

Muanalifah~\cite{Muanalifah2022} generalized this problem and considered the two-sided tropical discrete logarithm problem.
\begin{problem}
    Given matrices $W_1, W_2, V,$ and $C$ over the max-plus algebra of appropriate dimensions, find $x_1, x_2 \in \mathbb{Z}_{\geq 0}$ such that $W_1^{\otimes x_1} \otimes V \otimes W_2^{\otimes x_2} = C$.
\end{problem}

Subsequently, Alhussaini, Collett, and Sergeev~\cite{AlhussainiCollettSergeev2024} considered the tropical discrete logarithm problem with a shift and the tropical two-sided discrete logarithm problem with a shift.
\begin{problem}
    Given matrices $V, W,$ and $C$ of appropriate dimensions, find $s \in \mathbb{R}$ and $x \in \mathbb{Z}_{\geq 0}$ such that $s \otimes V \otimes W^{\otimes x} = C$.
\end{problem}
\begin{problem}
    Given matrices $W_1, W_2, V,$ and $C$ of appropriate dimensions, find $x_1, x_2 \in \mathbb{Z}_{\geq 0}$ and $s \in \mathbb{R}$ such that $s \otimes W_1^{\otimes x_1} \otimes V \otimes W_2^{\otimes x_2} = C$.
\end{problem}

Their approach to solving these problems was based on the CSR expansion of tropical matrix powers.

It is natural to increase the number of matrices and consider the tropical versions of the knapsack problem and the subset sum problem.

\begin{problem}\label{TropicalKnapsack}
    Given matrices $W_1, \ldots, W_n,$ and $C$ of size $k \times k$ over a tropical semiring, decide if $W_1^{\otimes x_1} \otimes \ldots \otimes  W_n^{\otimes x_n} = C$ for some non-negative integers $x_1, \ldots, x_n$.
\end{problem}

\begin{problem}\label{TropicalSubsetSum}
    Given matrices $W_1, \ldots, W_n$, and $C$ of size $k \times k$ over a tropical semiring, decide if $W_1^{\otimes x_1} \otimes \ldots \otimes  W_n^{\otimes x_n} = C$ for some $x_1, \ldots, x_n \in \{0, 1\}$.
\end{problem}

In this paper, we consider two algebraic structures: the set of matrices with non-negative integer entries of size $k\times k$ over the max-plus algebra with matrix multiplication and the set of matrices with positive integer entries of size $k\times k$ over the max-times algebra with matrix multiplication. 
We denote these structures by $\MP$ and $\MT$ respectively.

In this paper, we prove Theorem \ref{Th11} about the $\NP$-completeness of these four problems $\SSMPP$, $\SSMTP$, $\KMPP$, and $\KMTP$.
Theorem \ref{Th2} shows that there exist pseudo-polynomial algorithms to solve these problems. 
Also, we prove that there exist polynomial generic algorithms for $\SSMTP$ and $\KMTP$ (see Theorem~\ref{Th3}).

The remainder of this paper is structured into five parts. In Section 2, we recall some definitions from tropical algebra and generic-case complexity theory. Section 3 contains the proof of Theorem~\ref{Th11}. In Section 4, we prove Theorem~\ref{Th2}.
The next section contains the proof of Theorem~\ref{Th3}.
The final section gives a conclusion to our work and a list of open questions.

\section{Preliminaries}

In the first part of this section, we discuss tropical semirings and tropical matrix semirings, paying attention to the max-plus and max-times matrix algebras. In the second part of the section, we discuss how the size of an instance of a problem can be defined. 
In the last part, we recall the definition of the asymptotic density of a set and the definition of a generic algorithm.

Let $S$ be a set, and $a$ be a number. In this paper, we denote the set $\{x \in S : x \geq a\}$ by $S_{\geq a}$ and the set $\{x \in S : x > a\}$ by $S_{> a}$. For example, $\mathbb{Z}_{\geq 0}$ is the set of all non-negative integers, and $\mathbb{Z}_{\geq 1}$ is the set of all positive integers.
Also, we denote $S \cup \{\infty\}$ by $\overline{S}$ and $S \cup \{-\infty\}$ by $\underline{S}$.

A \textit{semiring} is a ring without the requirement that each element must have an additive inverse.
Two famous examples of semirings are the max-plus algebra and min-plus algebra. The \textit{max-plus algebra} is the set $\underline{\mathbb{R}}$ equipped with the operations 
$$
    x \oplus y = \max(x, y) \quad \text{and} \quad x \otimes y = x + y.
$$ 
The \textit{min-plus algebra} is the set $\overline{\mathbb{R}}$ equipped with the operations 
$$
    x \oplus y = \min(x, y) \quad \text{and} \quad  x \otimes y = x + y.
$$
These semirings are idempotent and commutative.
These two semirings are known as tropical algebras. They have been widely studied and have many applications. For more information, we refer the reader to~\cite{Butkovic2010}.

Some researchers studied the min-times and max-times algebras, where one of the operations is the multiplication of numbers, and the other is either $\min$ or $\max$~\cite{Shitov2016, SulandraIsnia2021, Durcheva2014, GaubertSergeev2008}. The domain of the \textit{max-times} algebra is $\mathbb{R}_{\geq 0}$, and the operations are 
$$
    x \oplus y = \max(x, y) \quad \text{and} \quad x \otimes y = x \cdot y.
$$
The domain of the \textit{min-times} algebra is $\overline{\mathbb{R}}_{> 0}$ and the operations are 
$$
    x \oplus y = \min(x, y) \quad \text{and} \quad x \otimes y = x \cdot y.
$$

One can use other sets closed under $\oplus$ and $\otimes$ as domains to define similar structures. For example, in case of the max-plus algebra, one can consider $\mathbb{Z}$ or $\mathbb{Z}_{\geq 0}$ instead of $\underline{\mathbb{R}}$.

We will need the following algebras: $\ZP = \left<\mathbb{Z}_{\geq 0}, {+}\right>$, $\ZT = \left<\mathbb{Z}_{\geq 1}, {\cdot}\right>$, $\ZMP = \left<\mathbb{Z}_{\geq 0}, \oplus, \otimes\right>$, where $\oplus$ is $\max$ and $\otimes$ is ${+}$, and  $\ZMT = \left<\mathbb{Z}_{\geq 1}, \oplus, \otimes\right>$, where $\oplus$ is $\max$ and $\otimes$ is ${\cdot}$.

Let $S$ be a set. We denote the set of matrices of size $k \times k$ with entries in $S$ by $\Mat_k(S)$. 
In this paper, for a matrix $A$, we often use $a_{ij}$ to refer to the element at the $i$-th row and the $j$-th column of the matrix $A$.

Let $\mathcal{S} = \left\langle S, \oplus, \otimes\right\rangle$ be an algebra with two binary operations, and let $k$ be a positive integer. The set $\Mat_k(S)$ can be equipped with addition $\oplus$ and multiplication $\otimes$ using the following formulae:
$$
    (A \oplus B)_{ij} = a_{ij} \oplus b_{ij},
$$
$$
    (A \otimes B)_{ij} = a_{i1} \otimes b_{1j} \oplus \dots \oplus a_{ik} \otimes b_{kj}.
$$
If $\mathcal{S}$ is an idempotent semiring, then the obtained algebra is also an idempotent semiring.

We denote an element $a$ of the semiring $\mathcal{S}$ raised to the $n$-th power by $a^{\otimes n}$. 

In this paper, we consider the following two algebras:
$\MP = \left<\Mat_k(\mathbb{Z}_{\geq 0}), \otimes\right>$, where $\otimes$ is defined using the maximum and the addition of numbers, and $\MT = \left<\Mat_k(\mathbb{Z}_{\geq 1}), \otimes\right>$, where $\otimes$ is defined using the maximum and the multiplication of numbers.

Note that these semigroups do not have the identity matrices. 
If $x_i = 0$ in an expression $W_1^{\otimes x_1} \otimes \ldots \otimes  W_n^{\otimes x_n}$, 
then it means that the matrix $W_i$ is absent in the product.

We have a partial order on $\Mat_k(\mathbb{Z}_{{\geq 0}})$ defined by the following rule:
$$
    A \leq B \text{ iff } a_{ij} \leq b_{ij} \text{ for all } i, j \in \{1, \ldots, k\}.
$$

Let us describe how to solve a linear matrix equation
\begin{equation}\label{MatLinEq}
A \otimes X = B
\end{equation}over
$\MP$ and $\MT$.

We denote the set of all solutions to an equation $F(X_1, \ldots, X_n) = G(X_1, \ldots, X_n)$
over $\mathcal{S}$ by $\V_{\mathcal{S}}(F(X_1, \ldots, X_n) = G(X_1, \ldots, X_n))$ or 
just by $\V(F(X_1, \ldots, X_n) = G(X_1, \ldots, X_n))$ if this does not cause misunderstandings.

Every linear matrix equation~(\ref{MatLinEq}) can be rewritten as a one-sided system of linear equations:
$$
    \begin{array}{lll}
        a_{11} \otimes x_{11} \oplus a_{12} \otimes x_{21} \oplus ... \oplus a_{1k} \otimes x_{k1}               & = & b_{11}, \\
        \qquad a_{11} \otimes x_{12} \oplus a_{12} \otimes x_{22} \oplus ... \oplus a_{1k} \otimes x_{k2}        & = & b_{12}, \\
        \hdotsfor{3} \\
        \qquad\qquad a_{11} \otimes x_{1k} \oplus a_{12} \otimes x_{2k} \oplus ... \oplus a_{1k} \otimes x_{kk}  & = & b_{1k}, \\
        a_{21} \otimes x_{11} \oplus a_{22} \otimes x_{21} \oplus ... \oplus a_{2k} \otimes x_{k1}               & = & b_{21}, \\
        \qquad a_{21} \otimes x_{12} \oplus a_{22} \otimes x_{22} \oplus ... \oplus a_{2k} \otimes x_{k2}        & = & b_{22}, \\
        \hdotsfor{3} \\
        \qquad\qquad a_{21} \otimes x_{1k} \oplus a_{22} \otimes x_{2k} \oplus ... \oplus a_{2k} \otimes x_{kk}  & = & b_{2k}, \\
        \hdotsfor{3}\\
        a_{k1} \otimes x_{11} \oplus a_{n2} \otimes x_{21} \oplus ... \oplus a_{kk} \otimes x_{k1}               & = & b_{k1}, \\
        \qquad a_{k1} \otimes x_{12} \oplus a_{k2} \otimes x_{22} \oplus ... \oplus a_{kk} \otimes x_{k2}        & = & b_{k2}, \\
        \hdotsfor{3} \\
        \qquad\qquad a_{k1} \otimes x_{1k} \oplus a_{k2} \otimes x_{2k} \oplus ... \oplus a_{kk} \otimes x_{kk}  & = & b_{kk}.
    \end{array}
$$

The matrix~$X^{*} = X^{*}(A, B)$ such that $x^{*}_{ij} = \min\{b_{lj} \otimes a_{li}^{\otimes {-1}}\}_{l}$ is called the \textit{principal solution} to the equation~(\ref{MatLinEq}).
Note that $x^{*}_{ij} = \min\{b_{lj} - a_{li}\}_{l}$ for $\MP$, and $x^{*}_{ij} = \min\{b_{lj} / a_{li}\}_{l}$ for $\MT$.
Actually, the principal solution may not be a solution to~(\ref{MatLinEq}).
The following lemma is an adaptation of Theorem 3.1.1 in~\cite{Butkovic2010}. 

\begin{lemma}\label{Lemma10}
    Let $\mathcal{S}$ be $\MP$ or $\MT$.
    Let $A, B, X$ be matrices of size $k \times k$. 
    $X$ is a solution to $A \otimes X = B$ over $\mathcal{S}$ if and only if $X \leq X^{*}(A, B)$ and 
    $$
        \bigcup_{(i, j) \;:\; x_{ij} = x^{*}_{ij}} M_{ij}(A, B) = \{1, \ldots, k\} \times \{1, \ldots, k\},
    $$
    where $M_{ij}(A, B) = \{(l, j) : x^{*}_{ij} = b_{lj} \otimes a_{li}^{\otimes {-1}}\}$.
\end{lemma}

This lemma gives us a way to enumerate all the solution to $A \otimes X = B$.

Now, let us define the size of an integer, a matrix, and an instance of a problem.

Let $a$ be a non-negative number.
We can consider two sizes of $a$. The first one is the number of bits that we need to write this number using the binary numeral system:
$$
    \size_2(a) = 
    \left\{
    \begin{array}{ll}
        \lfloor \log_2(a) \rfloor + 1 & \text{ if } a > 0, \\
        0                             & \text{otherwise}.
    \end{array}
    \right.
$$

For $a \geq 1$, if $\size(a) = s$, then $2^{s - 1}\leq a < 2^s$. Therefore, 
\begin{equation}\label{NumSize}
    |\{a \in \mathbb{Z}_{\geq 1} : \size_2(a) = s \}| = 2^{s - 1}.
\end{equation}

Note that for positive integers $a, b$,
\begin{equation}\label{sizeOfProduct}
    \size_2(a) + \size_2(b) - 1 \leq \size_2(ab) \leq \size_2(a) + \mathrm{size}(b).
\end{equation}

Indeed, let $\size_2(a) = s$ and $\size_2(b) = t$. Then $2^{s-1} \leq a < 2^s$ and $2^{t-1} \leq b < 2^t$. Therefore, $2^{s - 1 + t - 1} \leq ab < 2^{s+t}$. This implies $s + t - 1 \leq\size_2(ab) \leq s + t$.

The second size is the number of symbols that we need to write this number using the unary numeral system:
$$
    \size_1(a) = a.
$$
 
Let $A$ be a matrix of size $k \times k$.
The matrix can be represented as a list of $k^2$ numbers separated with $k^2 - 1$ delimiters.
Therefore,
\begin{equation}\label{sizeMatrix}
    \size_b(A) =  \sum_{i = 1}^k\sum_{j = 1}^k \size_b(a_{ij}) + k^2 - 1.
\end{equation}

Let $I = (w_1, \ldots, w_n, c)$ be an input of $\mathbf{SSP}(\mathcal{S})$ or $\mathbf{KP}(\mathcal{S})$. 
Then the size of $I$ is 
$$
    \size_b(I) = \sum_{i = 1}^n \size_b(w_i) + \size_b(c) + n.
$$

If $\mathcal{S}$ is a matrix semigroup, then the size of an input $I = (W_1, \ldots, W_n, C)$ can be rewritten as
$$
    \size_b(I) = \sum_{l = 1}^n \sum_{j = 1}^k\sum_{l = 1}^k \size_b(w_{lij}) + \sum_{i = 1}^k\sum_{j = 1}^k \size_b(c_{ij}) + (n + 1)k^2 - 1.
$$
Note that we consider the number $k$ as a parameter of a problem. Hence, $k$ can occur in exponents in polynomials.

Let $U$ be a set. \textit{A stratification} of $U$ is a sequence $\{U_m\}_m$ of non-empty finite subsets $U_m \subseteq U$ such that $\bigcup_m U_m = U$. 
For a subset $A \subseteq U$ and a stratification~$\{U_m\}_m$, the limit $$\rho(A) = \lim_{m \to \infty} \frac{|A \cap U_m|}{|U_m|}$$ (if it exists) is called the \textit{asymptotic density} of $A$ with respect to the stratification~$\{U_m\}_m$.
If $\rho(A) = 1$, we say that $A$ is \textit{generic}.
If $\rho(A) = 0$, we say that $A$ is \textit{negligible}.

An algorithm $\mathcal{A}\colon U \to V \cup \{?\}$ is called \textit{generic} if $\mathcal{A}$ stops on every input $I \in U$, and $\{I \in U : \mathcal{A}(I) \neq {?}\}$ is a generic set.
Here, the answer ``?'' means ``don't know''.

A decision problem $A \subseteq U$ is \textit{decidable generically in polynomial time} if there is a polynomial generic algorithm computing the indicator function of $A$.

\section{\texorpdfstring{$\NP$}{NP}-completeness of the problems}

In this section, we prove the theorems about the $\NP$-completeness of $\SSMPP$, $\SSMTP$, $\KMPP$, and $\KMTP$.
In order to do this, we need the following lemmas.

\begin{lemma}\label{Lemma11}
    $\SSMPP$ is in $\NP$.
\end{lemma}

\begin{proof}
    To prove that a problem is in $\NP$, we need to show that for every yes-instance $I$, there is a certificate of size polynomial in~$\size_2(I)$ that can be checked in time polynomial in~$\size_2(I)$. 

    It is easy to see that $\SSMPP$ is in $\NP$. 
    Indeed, let $I = (W_1, \ldots, W_n, C)$ be an input of the problem. A vector $(x_1, \ldots, x_n) \in \{0, 1\}^n$ such that 
    \begin{equation}\label{WWWEqC}
        W_1^{\otimes x_1} \otimes \ldots \otimes  W_n^{\otimes x_n} = C
    \end{equation}
    is a certificate. The size of this vector is polynomial in $\size_2(I)$. Also, the equality (\ref{WWWEqC}) can be checked in time polynomial in $\size_2(I)$. 
\end{proof}

\begin{lemma}\label{Lemma12}
    $\SSMTP$ is in $\NP$.
\end{lemma}

\begin{proof}
    The proof is similar to the proof of Lemma~\ref{Lemma11}.
\end{proof}

\begin{remark}\label{Remark1}
Let $A = (a_{ij})$ and $B = (b_{ij})$ be matrices in~$\Mat_k(\mathbb{Z}_{{\geq 0}})$.
Consider $\MP$. Note that 
$$(A \otimes B)_{ij} = a_{i1} \otimes b_{1j} \oplus \dots \oplus a_{ik} \otimes b_{kj} \geq a_{ij}.$$ 
Therefore, the sequence $A, A^{\otimes 2}, A^{\otimes 2}, \ldots$ is non-decreasing. Moreover, if $A \neq O$, then the maximal $x$ such that $A^{\otimes x} \leq B$ cannot be greater than $2 \cdot \max\{b_{ij}\}_{ij}$. 
\end{remark}

\begin{remark}\label{Remark2}
Consider $\MT$. Similarly, if $A = (a_{ij})$ is in $\Mat_k(\mathbb{Z}_{{\geq 1}})$, then the sequence 
$A, A^{\otimes 2}, A^{\otimes 2}, \ldots$ is non-decreasing. Moreover, if there is $a_{ij} > 1$, then
the maximal $x$ such that $A^{\otimes x} \leq B$ cannot be greater than $2 \cdot \max\{b_{ij}\}_{ij}$ as well.
\end{remark}

\begin{lemma}\label{Lemma13}
$\KMPP$ is in $\NP$.
\end{lemma}

\begin{proof}
Let $I = (W_1, \ldots, W_n, C)$ be an input of the problem.
To prove that $\KMPP$ is in $\NP$, let us demonstrate that if~(\ref{WWWEqC}) has a solution, then there is a solution of size polynomial in $\size_2(I)$.

If $W_i = O$, then without loss of generality we may assume that the corresponding $x_i$ is less than or equal to 1.
Using Remark~\ref{Remark1}, we get that the size of each $x_i$ is bounded by $2\size_2(I)$.
Therefore, the size of the solution is bounded by $2\size_2(I)^2$, which is a polynomial in $\size_2(I)$.

Since tropical matrix multiplication and tropical matrix exponentiation can be performed in time polynomial in $\size_2(I)$ and all the 
intermediate numbers during computing cannot be greater than $\max\{c_{ij}\}_{ij}$, the problem is in $\NP$.
\end{proof}

\begin{lemma}\label{Lemma14}
$\KMTP$ is in $\NP$.
\end{lemma}

\begin{proof}
Taking into account Remark~\ref{Remark2},
the proof is similar to the proof of Lemma~\ref{Lemma13}.
\end{proof}

\begin{lemma}\label{Lemma21}
$\SSMPP$ is $\NP$-hard.
\end{lemma}

\begin{proof}
Let us demonstrate that we can reduce in polynomial time the subset sum problem for non-negative integers $\SSP$ to $\SSMPP$. Note that $\SSP$ is $\NP$-complete~\cite[Problem SP13]{GareyJohnson1979}, \cite[Problem 18 (Knapsack)]{Karp1972}.
Define a function $f\colon \mathbb{Z}_{\geq 0} \to \Mat_k(\mathbb{Z}_{{\geq 0}})$ as follows:
\begin{equation}\label{FDef}
f(a) = 
\begin{pmatrix}
a      & \cdots & a     \\
\vdots & \ddots & \vdots \\
a      & \cdots & a      
\end{pmatrix}
.
\end{equation}

Let $(w_1, \ldots, w_n, c)$ be an input of $\SSP$;
then, we consider $(f(w_1), \ldots, f(w_n), f(c))$ as an input for $\SSMPP$.

Note that $x_1 a_1 + \ldots + x_n a_n = c$ if and only if
$f(a_1)^{\otimes x_1} \otimes \ldots \otimes f(a_n)^{\otimes x_n} = f(c)$ because
$$
f(a)^{\otimes x} = 
\begin{pmatrix}
a      & \cdots & a \\
\vdots & \ddots & \vdots \\
a      & \cdots & a      
\end{pmatrix}^{\!\!\otimes x} 
= 
\begin{pmatrix}
xa      & \cdots & xa \\
\vdots & \ddots & \vdots \\
xa      & \cdots & xa      
\end{pmatrix}
$$
and
\begin{equation*}
f(a) \otimes f(b) = 
\begin{pmatrix}
a      & \cdots & a \\
\vdots & \ddots & \vdots \\
a      & \cdots & a      
\end{pmatrix}
\otimes
\begin{pmatrix}
b      & \cdots & b \\
\vdots & \ddots & \vdots \\
b      & \cdots & b      
\end{pmatrix} 
=
\begin{pmatrix}
a+b      & \cdots & a+b \\
\vdots & \ddots & \vdots \\
a+b      & \cdots & a+b      
\end{pmatrix}.
\end{equation*}
Therefore, $\SSMPP$ is $\NP$-hard. 
\end{proof}

\begin{lemma}\label{Lemma22}
$\SSMTP$ is $\NP$-hard.
\end{lemma}

\begin{proof}
To demonstrate that $\SSMTP$ is $\NP$-hard, we can reduce in polynomial time the subset product problem $\SPP$ to this problem. $\SPP$ is known to be $\NP$-complete~\cite[Problem SP14]{GareyJohnson1979}.
Define a function $f\colon \mathbb{Z}_{\geq 1} \to \Mat_k(\mathbb{Z}_{{\geq 1}})$ by Formula~(\ref{FDef}).

For an input $(w_1, \ldots, w_n, c)$ of $\SPP$, consider $(f(w_1), \ldots, f(w_n), f(c))$ as an input for $\SSMTP$.

Note that $a_1^{x_1} \cdot \ldots \cdot  a_n^{x_n} = c$ if and only if
$f(a_1)^{\otimes x_1} \otimes \ldots \otimes f(a_n)^{\otimes x_n} = f(c)$ because
$$
f(a)^{\otimes x} = 
\begin{pmatrix}
a      & \cdots & a \\
\vdots & \ddots & \vdots \\
a      & \cdots & a      
\end{pmatrix}^{\!\!\otimes x} 
= 
\begin{pmatrix}
a^x      & \cdots & a^x \\
\vdots & \ddots & \vdots \\
a^x      & \cdots & a^x      
\end{pmatrix}
$$
and
\begin{equation*}
f(a) \otimes f(b) = 
\begin{pmatrix}
a      & \cdots & a \\
\vdots & \ddots & \vdots \\
a      & \cdots & a      
\end{pmatrix}
\otimes
\begin{pmatrix}
b      & \cdots & b \\
\vdots & \ddots & \vdots \\
b      & \cdots & b      
\end{pmatrix} 
=
\begin{pmatrix}
a\cdot b      & \cdots & a\cdot b \\
\vdots & \ddots & \vdots \\
a\cdot b      & \cdots & a\cdot b      
\end{pmatrix}.
\end{equation*}

Hence, this problem is $\NP$-hard.
\end{proof}

\begin{lemma}\label{Lemma23}
$\KMPP$ is $\NP$-hard.
\end{lemma}

\begin{proof}
To prove that $\KMPP$ is $\NP$-hard,
we can use the reduction described in the proof of Lemma~\ref{Lemma21} to
 reduce the knapsack problem for non-negative integers $\KP$ to $\KMPP$. $\KP$ is known to be $\NP$-complete~\cite[Proposition 4.1.1]{Haase2012}.
\end{proof}

Now, we consider the knapsack problem for positive integers with multiplication $\MKP$.

\begin{lemma}\label{LemmaKPZNPComplete}
$\MKP$ is $\NP$-complete.
\end{lemma}

\begin{proof}
To demonstrate that this problem is $\NP$-hard, we reduce the exact cover by 3-sets problem $\XThreeC$ in polynomial time to this problem. It is known that $\XThreeC$ is $\NP$-complete~\cite[Problem SP2]{GareyJohnson1979}.
Our proof is a modification of Theorem~8.8 in \cite{Moret1997} about the $\NP$-completeness of $\SPP$.

We use prime encoding. Let a set $A = \{a_1, a_2, \ldots, a_{3m}\}$ and a collection
$C = \{c_1, \ldots, c_n\}$ of $3$-element subsets of $A$ be an input of $\XThreeC$.
Generate the first $3m + n$ primes ordered in the increasing order $p_1, p_2, \ldots, p_{3m}, p_{3m+1}, \ldots p_{3m+n}$. Denote $p_{3m+1}, \ldots p_{3m+n}$ by $q_1, \ldots q_n$.
For every $c_i = \{a_{i_1}, a_{i_2}, a_{i_3}\}$, 
$i \in \{1, \ldots, n\}$, denote the primes $p_{i_1}, p_{i_2}, p_{i_3}$ with the corresponding indexes 
$i_1, i_2, i_3$ by $p_{i1}, p_{i2}, p_{i3}$.
Now, consider the following input of $\MKP$:
\begin{equation}\label{Lemma8Inst}
(p_{11}p_{12}p_{13}q_1,\, q_1,\, p_{21}p_{22}p_{23}q_2,\, q_2, \ldots,\, p_{n1}p_{n2}p_{n3}q_n,\, q_n,\, p_1p_2\cdots p_{3m}q_1\cdots q_n).
\end{equation}
Let this be a yes-instance, i.e. there is a vector $(x_1, y_1, x_2, y_2, \ldots x_n, y_n)$ such that
$$
(p_{11}p_{12}p_{13}q_1)^{x_1} \cdot q_1^{y_1} \cdot (p_{21}p_{22}p_{23}q_2)^{x_2} \cdot q_2^{y_2} \cdots  (p_{n1}p_{n2}p_{n3}q_n)^{x_n} \cdot q_n^{y_n} = p_1p_2\cdots p_{3m}q_1\cdots q_n.
$$
This implies that
$$
(p_{11}p_{12}p_{13})^{x_1} \cdot q_1^{x_1 + y_1} \cdot (p_{21}p_{22}p_{23})^{x_2} \cdot q_2^{x_2 + y_2} \cdots (p_{n1}p_{n2}p_{n3})^{x_n} \cdot q_n^{x_n + y_n} = p_1p_2\cdots p_{3m}q_1\cdots q_n.
$$
Note that $x_i + y_i = 1$ because $q_i$ can only be used once and only once. Therefore,
 $x_i \in \{0, 1\}$ and $y_i \in \{0, 1\}$. 
Hence,
\begin{equation}\label{Lemma8Fml}
(p_{11}p_{12}p_{13})^{x_1} \cdot (p_{21}p_{22}p_{23})^{x_2} \cdots (p_{n1}p_{n2}p_{n3})^{x_n} = p_1p_2\cdots p_{3m},
\end{equation}
where $x_i \in \{0, 1\}$.
Because of the fundamental theorem of arithmetic, 
the collection 
$$C' = \{c_i \in C : x_i = 1\}$$
is an exact cover for $A$.
Therefore, $(A, C)$ is a yes-instance of $\XThreeC$.

Conversely, assume that $(A, C)$  is a yes-instance of $\XThreeC$, and
$C' \subseteq C$ be an exact cover for $A$.
Consider the vector $(x_1, \ldots, x_n)$ such that 
$x_i = [c_i \in C']$, where $[\cdot]$ is the Iverson bracket.
Then
$$
(p_{11}p_{12}p_{13}q_1)^{x_1} \cdot q_1^{1 - x_1} \cdot (p_{21}p_{22}p_{23}q_2)^{x_2} \cdot q_2^{1 - x_2} \cdots (p_{n1}p_{n2}p_{n3}q_n)^{x_n} \cdot q_n^{1 - x_n} = p_1p_2\cdots p_{3m}q_1\cdots q_n.
$$
Therefore, (\ref{Lemma8Inst}) is a yes-instance as well.

Now, let us explain why this reduction is polynomial-time.
Observe that the largest number $p_1p_2\cdots p_{3m}q_1\cdots q_n$ can be computed in time $O(n \log q_n)$. Finding the $i$-th prime can be performed in $O(p_i)$ by the brute-force method of successive divisions.
From number theory, we know that $p_i$ is $O(i^2)$. Using this information, we conclude that the reduction runs in polynomial time.

Note that $\MKP$ is in $\NP$. Indeed, let $I = (w_1, \ldots, w_n, c)$ be an input. If $w_i = 1$, then we may assume that $x_i \leq 1$. 
If $w_i \geq 2$, then $x_i \leq \log_2{c}$. Hence, the size of $(x_1, \ldots, x_n)$ such that $w_1^{x_1} \ldots w_k^{x_n} = c$ is polinomial in $\size_2(I)$. Therefore, the problem is $\NP$-complete.
\end{proof}

\begin{lemma}\label{Lemma24}
$\KMTP$ is $\NP$-hard.
\end{lemma}

\begin{proof}
From Lemma~\ref{LemmaKPZNPComplete}, it follows that $\mathbf{KP}(\mathcal{N}_{\times})$ is $\NP$-complete.
Using the reduction described in the proof of Lemma~\ref{Lemma22}, we can reduce $\mathbf{KP}(\mathcal{N}_{\times})$ to $\KMTP$ in polynomial time.
\end{proof}

To prove that a problem is $\NP$-complete, we need to prove that it is in $\NP$ and is $\NP$-hard.
Therefore, combining the lemmas proved above, we obtain the following theorem, which is the main result of this section.

\begin{theorem}\label{Th11}
$\SSMPP$, $\SSMTP$, $\KMPP$, and $\KMTP$ are $\NP$-com\-plete.
\end{theorem}

\begin{proof}
This follows from Lemmas~\ref{Lemma11}, \ref{Lemma12}, \ref{Lemma13}, \ref{Lemma14}, \ref{Lemma21}, \ref{Lemma22}, \ref{Lemma23}, and \ref{Lemma24}. 
\end{proof}

\section{Pseudo-polynomial algorithms to solve the problems}

In this section, using the dynamic programming approach, we demonstrate that there are pseudo-polynomial algorithms to solve $\SSMPP$, $\SSMTP$, $\KMPP$, and $\KMTP$.
Recall that an algorithm is pseudo-polynomial if its time complexity is polynomial in $\size_1(I)$.

Let $\mathcal{S}$ be either $\MP$ or $\MT$.
Consider Algorithm~\ref{ALgo1} to solve $\mathbf{SSP}(\mathcal{S})$. 
This algorithm returns \textbf{true} if and only if there exists a vector $(x_1, \ldots, x_n) \in \{0, 1\}^n$ such that $W_1^{\otimes x_1} \otimes \ldots \otimes  W_n^{\otimes x_n} = C$ over $\mathcal{S}$.
This algorithm is based on the dynamic programming approach. 
We use memoization~\cite{Cormen2022} to avoid calling this function with the same parameters more than once. In other words, when the function is called, we first check if the result for the given parameters was computed. If it was, we return the stored result. Otherwise, we compute the result, save it for future use, and return it.
Note that $(0, 0, \ldots, 0)$ is not a solution to
(\ref{WWWEqC}) because $\mathcal{S}$ does not have the identity matrix.

\begin{algorithm}
    \caption{}\label{ALgo1}
    \begin{algorithmic}[0]
        \State $Memo \gets \Call{EmptyDict}{ }$
        \Function{SolveSSP}{$W_1$, \ldots, $W_n$, $C$}
            \State $I \gets (W_1, \ldots, W_n, C)$
            \If{$Memo$.\Call{Has}{$I$}}
                \Return $Memo[I]$
            \EndIf
            \If {$n = 0$} 
                \State $Memo[I] \gets \textbf{false}$
                \State \Return $Memo[I]$
            \EndIf
            \If{$W_1 = C$ \textbf{or} \Call{SolveSSP}{$W_2, \ldots, W_n, C$}} 
                \State $Memo[I] \gets \textbf{true}$
                \State \Return $Memo[I]$   
            \EndIf
            \ForAll {$X \in \V_\mathcal{S}(W_1 \otimes X = C)$}
                \If {\Call{SolveSSP}{$W_2, \ldots, W_n, X$}}
                    \State $Memo[I] \gets \textbf{true}$
                    \State \Return $Memo[I]$   
                \EndIf
             \EndFor
             \State $Memo[I] \gets \textbf{false}$
             \State \Return $Memo[I]$
        \EndFunction
    \end{algorithmic}
\end{algorithm}

\begin{lemma}\label{LemmaAlgo1}
The time complexity of Algorithm~\ref{ALgo1} is polynomial in $\size_1(I)$, where $I = (W_1, \ldots, W_n, C)$.
\end{lemma}
\begin{proof}
Note that all the operations used in the body of $\Call{SolveSSP}{}$ are polynomial-time in~$\size_1(I)$.
Therefore, it is enough to show that the number of recursive calls is polynomial in ~$\size_1(I)$.
Since we use memoization, we need to compute the number of different inputs $(W_i, \ldots, W_n, Y)$ of this function.
First, note that the number of different tuples $(W_i, \ldots, W_n)$ is $n + 1$. Second, 
for every $Y$, there is a matrix $X$ such that $X \otimes Y = C$. Hence, $0 \leq y_{ij} \leq M$, where $M = \max\{c_{ij}\}_{ij}$. Therefore,
the number of different inputs is not greater than
$(n + 1) (M + 1)^{k^2}$. This expression is $O((\size_1(I))^{k^2 + 1})$. 
Therefore, the algorithm is polynomial-time in $\size_1(I)$.
\end{proof}

For the knapsack problem, consider Algorithm~\ref{ALgo2}.

\begin{algorithm}
    \caption{}\label{ALgo2}
    \begin{algorithmic}[0]
        \State $Memo \gets \Call{EmptyDict}{ }$
        \Function{SolveKP}{$W_1$, \ldots, $W_n$, $C$}  
            \State $I \gets (W_1, \ldots, W_n, C)$
            \If{$Memo$.\Call{Has}{$I$}}
                \Return $Memo[I]$
            \EndIf
            \If {$n = 0$} 
                \State $Memo[I] \gets \textbf{false}$
                \State \Return $Memo[I]$
            \EndIf
            \If{$W_1 = C$ \textbf{or} \Call{SolveKP}{$W_2, \ldots, W_n, C$}} 
                \State $Memo[I] \gets \textbf{true}$
                \State \Return $Memo[I]$   
            \EndIf
            \ForAll {$X \in \V_\mathcal{S}(W_1 \otimes X = C)$}
                \If {\Call{SolveKP}{$W_2, \ldots, W_n, X$} \textbf{or} $X \neq C$ \textbf{and} \Call{SolveKP}{$W_{1}, \ldots, W_n, X$}}
                    \State $Memo[I] \gets \textbf{true}$
                    \State \Return $Memo[I]$   
                \EndIf
             \EndFor
             \State $Memo[I] \gets \textbf{false}$
             \State \Return $Memo[I]$
        \EndFunction
    \end{algorithmic}
\end{algorithm}

\begin{lemma}\label{LemmaAlgo2}
The time complexity of Algorithm~\ref{ALgo2} is polynomial in $\size_1(I)$, where $I = (W_1, \ldots, W_n, C)$.
\end{lemma}
\begin{proof}
The proof is similar to the proof of Lemma~\ref{LemmaAlgo1}.
\end{proof}

\begin{theorem}\label{Th2}
$\KMPP$, $\KMTP$, $\SSMPP$, and $\SSMTP$
are weakly $\NP$-complete.
\end{theorem}
\begin{proof}
It follows from Lemmas~\ref{LemmaAlgo1} and~\ref{LemmaAlgo2} that
there are pseudo-polynomial algorithms to solve these problems.
\end{proof}

\section{Generic polynomial algorithms to solve \texorpdfstring{$\SSMTP$}{SSP(M\_{max, ×}\^{}k)} and \texorpdfstring{$\KMTP$}{KP(M\_{max, ×}\^{}k)}}

In this section, we prove that there are polynomial generic algorithms to solve $\SSMTP$ and $\KMTP$.
Let $U$ be the set of all tuples of the form $(W_1, \ldots, W_n, C)$, 
where $W_1, \ldots, W_n, C \in \Mat_k(\mathbb{Z}_{\geq 1})$ and $n \in \mathbb{Z}_{\geq 0}$.
When considering a generic algorithm, we need to define a stratification of $U$.
In this section, we consider the stratification $\{U_m\}_m$ of $U$,
where $U_m = \{I \in U : \size_2(I) = m\}$. 

Let $M_l = \{A \in \Mat_k(\mathbb{Z}_{\geq 1}) : \size_2(A) = l\}$.
Each matrix in $M_l$ can be encoded by a string with $k^2 - 1$ delimiters of length $l$ of the following form:
\begin{equation}\label{pattern}
1d_{11}d_{12}\ldots d_{1l_1} \text{\textvisiblespace} 1d_{21}d_{22}\ldots d_{2l_2}\text{\textvisiblespace}\ldots\text{\textvisiblespace}1d_{k^21}d_{k^22}\ldots d_{k^2l_{k^2}},
\end{equation}
where $d_{ij} \in \{0, 1\}$.
Note that $l_1 + l_2 + \ldots + l_{k^2} = l - 2k^2 - 1$. Hence,
when the positions of the delimiters are fixed, each pattern~(\ref{pattern}) 
encodes $2^{l - 2k^2 - 1}$
matrices. The number of these patterns can be found using the stars and bars method and is equal to $\binom{l - k^2}{k^2 - 1}$. Therefore,
$$
|M_l| = 2^{l - 2k^2 + 1} \binom{l - k^2}{k^2 - 1}.
$$

\begin{lemma}\label{LemmaGen13}
Let $X, Y \in \Mat_k(\mathbb{Z}_{\geq 1})$. 
Then, $\size_2(X) + \size_2(Y) \leq \size_2(X \otimes Y) + 2k^2 - 1$. 
\end{lemma}

\begin{proof}
Let $A = X \otimes Y$. Then, for every $i$ and $j$,
$$
x_{i1} \otimes y_{1j} \oplus x_{i2} \otimes y_{2j} \oplus \cdots \oplus x_{ik} \otimes y_{kj} = a_{ij}.
$$
In the other notation, 
$\max\{x_{il} \cdot y_{lj}\}_{l} = a_{ij}.$
Therefore, for each index $l$, $x_{il} \cdot y_{lj} \leq a_{ij}.$
This implies that for each $l$,
$\size_2(x_{il} \cdot y_{lj}) \leq \size_2(a_{ij})$.
Using~(\ref{sizeOfProduct}), we get 
$$
\size_2(x_{il}) + \size_2(y_{lj}) \leq \size_2(a_{ij}) + 1.
$$

Let us choose 
\begin{equation}\label{DefL}
l = (i + j) \bmod k + 1.
\end{equation}
Then,
$$
\size_2(x_{i, (i + j) \bmod k + 1}) + \size_2(y_{(i + j) \bmod k + 1, j}) \leq \size_2(a_{ij}) + 1.
$$
Summing this inequality for all $i$ and $j$, we obtain
$$
\sum_{i = 1}^k\sum_{j = 1}^k\size_2(x_{i, (i + j) \bmod k + 1}) + \sum_{i = 1}^k\sum_{j = 1}^k\size_2(y_{(i + j) \bmod k + 1, j}) \leq \sum_{i = 1}^k\sum_{i = 1}^k\size_2(a_{ij}) + k^2.
$$
Note that Formula~(\ref{DefL}) defines an rearrangement of $\{1, \ldots, k\}$ for each fixed $i$ and each fixed $j$. Therefore, the first double sum has all the elements of $X$,
and the second double sum has all the elements of $Y$:
$$
\sum_{i = 1}^k\sum_{j = 1}^k\size_2(x_{ij}) + \sum_{i = 1}^k\sum_{j = 1}^k\size_2(y_{ij}) \leq \sum_{i = 1}^k\sum_{i = 1}^k\size_2(a_{ij}) + k^2.
$$
Using Formula~(\ref{sizeMatrix}), we obtain
$$
\size_2(X) - k^2 + 1 + \size_2(Y) - k^2 + 1 \leq \size_2(A) - k^2 + 1 + k^2.
$$
This yields that 
\[
\size_2(X) + \size_2(Y)\leq \size_2(A) + 2k^2 - 1.
\qedhere
\]
\end{proof}

\begin{lemma}\label{Lemma100}
Let $S_{l, P(m)} = \{A \in M_l : |\!\V(X \otimes Y = A)| > P(m)\}$,
where $l \leq m$ and $P(m)$ is a polynomial. Then,
\begin{equation}
\label{FormulaLemmaGen14}
|S_{l, P(m)}| \leq \frac{m^{k^2 + 1}}{P(m)}
\cdot|M_l|.
\end{equation}
\end{lemma}

\begin{proof}
Let 
$$
K_l = \{(X, Y) \in (\Mat_k(\mathbb{Z}_{\geq 1}))^2 : \size_2(X \otimes Y) = l\}.
$$
Note that $P(m) \cdot |S_{l, P(m)}| \leq |K_l|$.
Indeed,
\begin{multline*}
|K_l| = \sum_{A \in M_l}|\{(X, Y) \in (\Mat_k(\mathbb{Z}_{\geq 1}))^2 : X \otimes Y = A\}| =
\sum_{A \in M_l}|\!\V(X \otimes Y = A)| = \\
\sum_{A \in M_l \,:\, |\!\V(X \otimes Y = A)| \leq P(m)}|\!\V(X \otimes Y = A)| +
\sum_{A \in M_l \,:\, |\!\V(X \otimes Y = A)| > P(m)}|\!\V(X \otimes Y = A)| \geq \\
\sum_{A \in M_l \,:\, |\!\V(X \otimes Y = A)| > P(m)}|\!\V(X \otimes Y = A)| \geq 
\sum_{A \in M_l \,:\, |\!\V(X \otimes Y = A)| > P(m)} P(m) =\\
P(m) \cdot |\{A \in M_l \,:\, |\!\V(X \otimes Y = A)| > P(m)\}|.
\end{multline*}
It follows from Lemma~\ref{LemmaGen13} that
$|K_l| \leq |L_{l}|$, where
$$
L_l = \{(X, Y) \in (\Mat_k(\mathbb{Z}_{\geq 1}))^2 : \size_2(X) + \size_2(Y) \leq l + 2k^2 - 1\}.
$$
Therefore,
\begin{equation}\label{PSKL}
P(m) \cdot |S_{l, P(m)}| \leq |K_l| \leq |L_{l}|.    
\end{equation}
Let
$$
N_r = \{(X, Y) \in (\Mat_k(\mathbb{Z}_{\geq 1}))^2 : \size_2(X) + \size_2(Y) = r\}.
$$
Note that
$$
|L_{l}| =
\sum_{r = 4k^2 - 2}^{l + 2k^2 - 1} |N_r| \leq (l - 2k^2 + 2) \cdot |N_{l + 2k^2 - 1}| \leq m \cdot |N_{l + 2k^2 - 1}|,
$$
and, using the stars and bars method again, we have
$$
|N_{l + 2k^2 - 1}| = 2^{l - 2k^2 + 1}\cdot \binom{l}{2k^2 - 1}.
$$
Hence, we have
\begin{multline*}
|L_{l}| \leq
m \cdot 2^{l - 2k^2 + 1}\cdot \binom{l}{2k^2 - 1} =
\frac{m \cdot 2^{l - 2k^2 + 1}\cdot l!}{(2k^2 - 1)!\cdot(l - 2k^2 + 1)!} = \\
\frac{m \cdot 2^{l - 2k^2 + 1}\cdot l!}{(2k^2 - 1)!\cdot (l - 2k^2 + 1)!} \cdot \frac{|M_l|}{|M_l|}= 
\frac{m \cdot 2^{l - 2k^2 + 1}\cdot l! \cdot (k^2 - 1)! \cdot (l - 2k^2 + 1)!}{2^{l - 2k^2 + 1}\cdot(2k^2 - 1)!\cdot(l - 2k^2 + 1)!\cdot(l - k^2)!} \cdot |M_l|= \\
\frac{m \cdot l! \cdot (k^2 - 1)!}{(2k^2 - 1)!\cdot(l - k^2)!} \cdot |M_l|\leq\\
\frac{m \cdot l!}{(l - k^2)!} \cdot |M_l|\leq
m \cdot l^{k^2} \cdot |M_l|\leq
 m^{k^2 + 1} \cdot |M_l|.
\end{multline*}
Combining this inequality and the inequality~(\ref{PSKL}), we obtain~(\ref{FormulaLemmaGen14}).
\end{proof}

\begin{lemma}\label{LemmaRm}
Let $R_{m, P(m)} = \{(W_1, \ldots, W_n, C) \in U_m : |\!\V(X \otimes Y = C)| > P(m)\}$, where $P(m)$ is a polynomial.
Then,
\begin{equation}\label{Formula101}
\frac{|R_{m, P(m)}|}{|U_m|} \leq \frac{m^{k^2 + 1}}{P(m)}.
\end{equation}
\end{lemma}

\begin{proof}
Let $Q_l = \{ (W_1, \ldots, W_n) \in (\Mat_k(\mathbb{Z}_{\geq 1}))^n: \size_2(W_1) + \ldots + \size_2(W_n) + n = l\}$. Note that $U_m$ can be presented as
$$U_m = \bigcup_{l = 2k^2 - 1}^m Q_{m - l} \times M_l.$$ Hence,
$|U_m|$ can be calculated using the following formula:
$$
|U_m| = \sum_{l = 2k^2 - 1}^m |Q_{m - l}| \cdot |M_l|. 
$$
Also, using Lemma~\ref{Lemma100}, we have
\begin{multline*}
|R_{m, P(m)}| = \sum_{l = 2k^2 - 1}^m |S_{l, P(m)}| \cdot |Q_{m - l}|
\leq
\sum_{l = 2k^2 - 1}^m 
\frac{m^{k^2 + 1}}{P(m)}
\cdot|M_l| \cdot |Q_{m - l}|
=\\
\frac{m^{k^2 + 1}}{P(m)}\cdot
\sum_{l = 2k^2 - 1}^m 
|M_l| \cdot |Q_{m - l}|
=
\frac{m^{k^2 + 1}}{P(m)} \cdot |U_m|.
\end{multline*}
Using this inequality, we obtain the inequality~(\ref{Formula101}).
\end{proof}

\begin{theorem}\label{Th3}
$\SSMTP$ and $\KMTP$ are decidable generically in polynomial time.
\end{theorem}
\begin{proof}
Consider $\SSMTP$. We modify Algorithm~\ref{ALgo1} as follows. 
Let $I = (W_1, \ldots, W_n, C)$ and $m = \size_2(I)$.
In the beginning of the function $\Call{SolveSSP}{}$,
we check the size of the dictionary $Memo$. If this size is greater than
$m^{k^2 + 3}$, then we stop the process and the algorithm returns ``?''.
Note that the time complexity of this algorithm is polynomial in $m$.

Consider the set 
$$R = \bigcup_{m}{R_{m, m^{k^2 + 2}}}.$$ From Lemma~\ref{LemmaRm},
 it follows that 
$$\rho(R) = \lim_{m \to \infty}\frac{|R_{m, m^{k^2 + 2}}|}{|U_m|} = 0.$$
Therefore, it is enough to show that for every $I \in U \setminus R$,
the algorithm returns \textbf{true} or \textbf{false}.
Let us estimate how many different tuples $(W_i, \ldots, W_n, Y)$ we can obtain for the input $I$. 
First, note that the number of different tuples $(W_i, \ldots, W_n)$
is $n + 1$. Second, note that for every $Y$, there is a matrix $X$ such that $X \otimes Y = C$. Therefore, the number of different $Y$ is not greater than
$|\!\V(X \otimes Y = C)|$. Since $I\in U \setminus R$, this number is not greater than
$m^{k^2 + 2}$. Hence, the number of different tuples $(W_i, \ldots, W_n, Y)$ is not greater than $m \cdot m^{k^2 + 2} = m^{k^2 + 3}$. Therefore, the algorithm considers all such tuples. 

Algorithm~\ref{ALgo2} to solve $\KMTP$ can be modified similarly.
\end{proof}

\section{Conclusion and questions for further study}
In this paper, we considered the knapsack problem and the subset sum problem for the algebras $\MP$ and $\MT$. We proved that these problems are $\NP$-complete. We showed that
there are pseudo-polynomial algorithms to solve these problems. Also, we showed that there
are polynomial generic algorithms to solve the knapsack problem and the subset sum problem for $\MT$. It would be interesting to know if there are such algorithms for $\MP$.
\begin{question}
Does there exist polynomial generic algorithms to solve $\SSMPP$ and $\KMPP$?
\end{question}

We proved our results for matrices without negative numbers and $-\infty$. It would be interesting to change our results for the semigroup of matrices of size $k \times k$ over the max-plus algebra
$\left<\underline{\mathbb{Z}}, \oplus, \otimes\right>$.
Denote this structure by $\RP$.   
\begin{question}
    Is $\mathbf{KP}(\RP)$ $\NP$-complete?
    Are there pseudo-polynomial or polynomial generic algorithms to solve $\mathbf{SSP}(\RP)$ and $\mathbf{KP}(\RP)$?
\end{question}

Also, we can consider matrices over the min-plus algebra $\left<\overline{\mathbb{Z}}, \oplus, \otimes\right>$. Denote this structure by $\QP$.  
\begin{question}
    Is $\mathbf{KP}(\QP)$ $\NP$-complete?
    Are there pseudo-polynomial or polynomial generic algorithms to solve $\mathbf{SSP}(\QP)$ and $\mathbf{KP}(\QP)$?
\end{question}

Alhussaini et al. considered the tropical discrete logarithm problem with a shift and the tropical two-sided discrete logarithm problem with a shift~\cite{AlhussainiCollettSergeev2024}.
It would be interesting to consider a shifted version of the knapsack problem.
\begin{problem}
    Given matrices $W_1, \ldots, W_n$, and $C$ of size $k \times k$ over a tropical semiring, decide if $s \otimes W_1^{\otimes x_1} \otimes \ldots \otimes  W_n^{\otimes x_n} = C$ for some nonnegative integers $x_1, \ldots, x_n$ and $s \in \mathbb{Z}$.
\end{problem}
Denote this problem for $\RP$ by $\textbf{SKP}(\RP)$.
\begin{question}
    Is $\mathbf{SKP}(\RP)$ $\NP$-complete? 
    Is there a polynomial or a polynomial generic algorithm to solve $\mathbf{SKP}(\RP)$?
\end{question}

\section*{Acknowledgments}

This research was supported in accordance with the state task of the IM SB RAS, project FWNF-2026-0033.

The authors have no competing interests to declare that are relevant to the content of this manuscript.

\renewcommand*{\bibfont}{\small}
\printbibliography
\end{document}